\documentclass[11pt]{article}
\usepackage{amssymb,amsmath,amsfonts}
\usepackage{stmaryrd}
\usepackage{color}
\usepackage[latin1]{inputenc}
\usepackage{multirow}
\usepackage{subfigure}
\usepackage{epsfig}
\usepackage{epic}
\usepackage{pstricks}
\usepackage{pst-node}
\usepackage{pstricks-add}

\def\BBox{\kern  -0.2cm\hbox{\vrule width 0.2cm height 0.2cm}}

\newtheorem{lemma}{Lemma}[section]
\newtheorem{theorem}{Theorem}[section]
\newtheorem{definition}{Definition}[section]

\newtheorem{proposition}{Proposition}[section]
\newtheorem{remark}{Remark}[section]

\title{An explicit formula for obtaining  $(q+1,8)$-cages and
others small regular graphs of girth 8}

\author{ M. Abreu$^{1}$, G. Araujo-Pardo$^{2}$, C. Balbuena$^{3}$,
D. Labbate$^{4}$
\thanks{ Research   supported by the Ministerio de Educaci\'on y Ciencia,
Spain, the European Regional Development Fund (ERDF) under project
MTM2008-06620-C03-02;   and under the Catalonian Government
project 1298 SGR2009. CONACyT-M\'exico under project 57371 and
PAPIIT-M\'exico under project 104609-3. \newline \footnotesize{\em
Email addresses:} marien.abreu@unibas.it (M. Abreu),~
garaujo@matem.unam.mx (G. Araujo), ~ m.camino.balbuena@upc.edu (C.
Balbuena), \, \, ~ labbate@poliba.it (D. Labbate)}
 \\[2ex]
$^1${\footnotesize Dipartimento di Matematica, Universit\`{a} degli Studi della
Basilicata,}\\
{\footnotesize Viale dell'Ateneo Lucano, I-85100 Potenza, Italy.} \\$^2$
{\footnotesize Instituto de Matem\'{a}ticas, Universidad Nacional Aut\'{o}noma de M\'{e}xico,} \\
{\footnotesize M\'{e}xico D. F., M\'exico }\\
$^3${\footnotesize Departament de Matem\`atica Aplicada III, Universitat
Polit\`ecnica de Catalunya, }\\
{\footnotesize Campus Nord, Edifici C2, C/ Jordi Girona 1 i 3 E-08034 Barcelona,
Spain.} \\
$^4${\footnotesize Dipartimento di Matematica, Politecnico di
Bari,  I-70125 Bari, Italy.}}

\date{}

\begin{document}
\maketitle
\begin{abstract}
Let $q$ be a prime power;  $(q+1,8)$-cages  have been
constructed as incidence graphs of a non-degenerate quadric surface in projective
4-space $P(4, q)$. The first contribution of this paper is a
construction of these graphs in an alternative way by means of an
explicit formula using graphical terminology. Furthermore by removing
some specific perfect dominating sets from a $(q+1,8)$-cage  we
derive  $k$-regular graphs of girth 8 for  $k= q-1$ and $k=q$, having the
smallest number of vertices  known so far.

\end{abstract}



\section{Introduction}

Throughout this paper, only undirected simple graphs without loops
or multiple edges are considered. Unless otherwise stated, we
follow  the book by Godsil and   Royle \cite{GR00} and the book by
Lint and Wilson \cite{LW94}  for terminology and definitions.

 Let $G$ be a graph with vertex set
$V=V(G)$ and edge set
  $E=E(G)$.  The \emph{girth}  of a graph $G$ is the number $g=g(G)$ of edges in a
smallest cycle. For every $v\in V$, $N_G(v)$ denotes the \emph{neighbourhood} of $v$,
that is, the set of all vertices adjacent to $v$. The \emph{degree} of a vertex $v\in V$ is the
cardinality of    $N_G(v)$.   A graph is called
\emph{regular} if all the vertices have the same degree. A  \emph{$(k,g)$-graph} is a  $k$-regular graph with girth $g$.  Erd\H os and Sachs
 \cite{ES63}     proved the existence of  $(k,g)$-graphs
 for all values of $k$ and $g$ provided that $k \ge 2$. Thus most work carried
 out has  focused on constructing a smallest one
 \cite{AFLN06,ABH10,BI73,B08,B09,B66,BMS95,E96,FH64,GH08,LU95,LUW97,M99,OW81,PBMOG04}.
A  \emph{$(k,g)$-cage} is a  $k$-regular graph with girth $g$ having the smallest possible
number of vertices. Cages have been  studied
intensely since they were introduced by
Tutte \cite{T47} in 1947.
Counting the numbers of vertices in the
distance partition with respect to a vertex yields a lower bound
$n_0(k,g)$ with the
precise form of the bound depending on whether $g$ is even or odd:

  \begin{equation}\label{lower} n_0(k,g) = \left\{ \begin{array}{ll} 1 + k + k(k-1) + \cdots
+ k(k-1)^{(g-3)/2} &\mbox{ if $g$ is odd};\\
2(1 +(k-1) +
\cdots + (k-1)^{g/2-1})&\mbox{ if $g$ is
even}.\end{array}\right.\end{equation}
 Biggs   \cite{B96}
 calls the \emph{excess} of a  $(k,g)$-graph $G$
the difference $|V(G)|-n_0(k,g)$. The    construction
of graphs with small excess is a difficult task. 
Biggs is the author of a report on distinct methods for constructing  cubic cages
  \cite{B98}.
More details about constructions of cages can be found in the survey by Wong
\cite{W82} or in the  book  by Holton and Sheehan \cite{HS93} or in  the more
recent dynamic cage survey by Exoo and Jajcay \cite{EJ08}.

A  $(k,g)$-cage with
 $n_0(k,g)$ vertices and even girth  exist only when $g\in \{4,6,8,12\}$ \cite{FH64}. If $g=4$ they are the complete bipartite graph $K_{k,k}$, and for $g=6,8,12$ these graphs are the incidence graphs of generalized $g/2$-gons of order $k-1$.  This is the main reason for $(k,g)$-cages with
 $n_0(k,g)$ vertices and even girth $g$  are called \emph{generalized polygon graphs}   \cite{B96}. In particular a $3$-gon of order $k-1$ is also known as a \emph{projective plane} of order $k-1$. The $4$-gons of order $k-1$ are called \emph{generalized quadrangles} of order $k-1$,
and, the $6$-gons of order $k-1$, \emph{generalized hexagons} of order $k-1$.  All these objets   are   known
to exist for all prime power values of $k-1$ \cite{B97,GR00,LW94}, and no example is known when $k-1$ is not a prime power.

In this article we focus on the case $g =8$. Let $q$ be a prime power.    Our main objective is to give an explicit construction of small $(k,8)$-graphs for $k=q-1,q$   and   $q+1 $. Next we present the contributions of this paper and in the following sections 
the corresponding proofs.

 $(q+1,8)$-cages  have been constructed by Benson
\cite{B66} as follows. Let $Q_4$ be a non-degenerate quadric
surface in projective 4-space $P(4, q)$. Define $G_8$ to be the
graph whose  vertices are the points and lines of $Q_4$, two  vertices
being joined if and only if they correspond to an incident
point-line pair in $Q_4$. Then $G_8$   is a   $(q +
1)$-regular graph of girth 8 with $n_0(q+1,8)$ vertices. The first contribution of this paper
is a construction of these graphs in an alternative way by means
of   an explicit formula given next.
\begin{definition}\label{gammaq}
Let $\mathbb{F}_q$ be a finite  field with $q\ge 2$ a prime power. Let $\Gamma_q=\Gamma_q[V_0,V_{1}]$ be a bipartite graph with vertex sets
$V_r=\{(a,b,c)_r, (q,q,a)_r: a\in \mathbb{F}_q \cup  \{q\} ,b, c \in \mathbb{F}_q \}$, $r=0,1$, and
edge set defined as follows:
$$ \begin{array}{l}\mbox{For all } a\in \mathbb{F}_q\cup \{q\}  \mbox{ and for all } b,c\in \mathbb{F}_q:\\[2ex]
N_{\Gamma_q}((a,b,c)_{1} )= \left\{\begin{array}{ll}
    \{(x,~ax+b,~a^2x+2ab+c)_{0}: x \in \mathbb{F}_q \}\cup \{(q,a,c)_{0} \}  &\mbox{ if }   a\in \mathbb{F}_q;
\\[2ex]
\{( c ,b,x )_{0}: x \in \mathbb{F}_q \}\cup \{(q,q,c)_{0} \}  &\mbox{
if }  a= q.
\end{array}\right.
\\
\mbox{}\\
N_{\Gamma_q}((q,q,a)_{1})= \{(q,a,x)_{0}: x \in \mathbb{F}_q \}\cup
\{(q,q,q)_{0} \}.
\end{array}
$$
\end{definition}

\begin{theorem}\label{main8}
The graph $\Gamma_q$ is a  $(q+1,8)$--cage on $2q^3+2q^2+2q+2$  vertices for each prime power $q$.
\end{theorem}


\begin{remark}\label{rem}
\begin{itemize}
\item[(i)] Let $\Gamma_q$ be  a  $(q+1,8)$-cage obtained in Theorem \ref{main8}. Using geometrical terminology we call the elements of $V_1$ \emph{lines} and the elements of $V_{0}$ \emph{points}. Then $\Gamma_q$ is the incidence graph of a   classical  generalized quadrangle $Q(4,q) $.
\item[(ii)] The edge set of a $(q+1,8)$-cage  $\Gamma_q$ obtained in Theorem \ref{main8}  can equivalently be expressed as follows:$$ \begin{array}{l}\mbox{For all } x\in \mathbb{F}_q\cup \{q\}  \mbox{ and for all }  y,z\in \mathbb{F}_q:\\[2ex]
N_{\Gamma_q}((x,y,z)_{0} )= \left\{\begin{array}{ll}
    \{(a,~y-ax,~a^2x-2ay+z)_{1}: a \in \mathbb{F}_q \}\cup \{(q,y,x)_1 \}  &\mbox{ if }   x\in \mathbb{F}_q;
\\[2ex]
\{( y ,a,z )_{1}: a \in \mathbb{F}_q \}\cup \{(q,q,y)_{1} \}  &\mbox{ if }  x= q.
\end{array}\right.
\\
\mbox{}\\
N_{\Gamma_q}((q,q,z)_{0})= \{(q,a,z)_{1}: a \in \mathbb{F}_q \}\cup
\{(q,q,q)_{1} \}; \\[1ex]
N_{\Gamma_q}((q,q,q)_{0})= \{(q,q,x)_{1}: x \in \mathbb{F}_q \cup\{q\}\}.
\end{array}
$$
\end{itemize}
\noindent Therefore, if $q$ is even, $2a=0$ for all $a\in
\mathbb{F}_q$  yielding that if  the partite sets $V_0$ and $V_{1}$ are
interchanged   the same graph $\Gamma_q$ is obtained. Equivalently, if
$q$ is even (in geometrical terminology) the corresponding
generalized quadrangle $Q(4,q) $ is said to be \emph{self-dual.}

\end{remark}

A bipartite graph is said to be \emph{balanced} if each partite set has the same number of vertices. Let   $q\ge 2$ be a prime  power. In
what follows we construct $(k,8)$-regular balanced
bipartite graphs for $k=q-1$  and $k = q$ with smallest known order.
We will use the following notation.
  Given an integer $t\ge 1$, a graph $G$ and a vertex $u\in V(G)$,  let  $N_G^t(u)=\{x\in V(G):
d_G(u,x)= t\}$, and   $N_G^t[u]=\{x\in V(G): d_G(u,x)\le t\}$,
where $d_G(u,x)$ denotes the distance between $u$ and $x$ in $G$. Note that  $N_G^1(u)=N_G (u)$.
A subset $U\subset V(G)$ is said to be   \emph{a perfect  dominating set of $G$} if for each vertex $x\in V(G)\setminus U$,  $|N_G(x)\cap U|=1$ \cite{HHS98}. Let $\Gamma_q=\Gamma_q[V_{0},V_{1}]$ be the $(q+1, 8)$-cage constructed in Theorem \ref{main8}. Suppose that $U$ is a  perfect dominating set of $\Gamma_q$, then $\Gamma_q-U$ is a $q$-regular graph of girth  8. Thus it is of interest to find the largest  perfect dominating set of $\Gamma_q$. In the following theorem we find perfect dominating sets of orders  $2(q^2+1)$, $ 2(q^2+3q+1)$
 for any prime power $q$,  and of order $2(q^2+4q+3)$  for even prime powers $q$.


\begin{theorem}\label{perfect}Let $q\ge 2$ be a  prime power  and   $\Gamma_q=\Gamma_q[V_{0},V_{1}]$ the $(q+1, 8)$-cage constructed in Theorem \ref{main8}.  The following sets are perfect dominating in  $\Gamma_q$:
\begin{itemize}
\item[(i)] $A=N_{\Gamma_q}^2[\alpha]\cup N_{\Gamma_q}^2[\beta] $ where $\alpha,\beta \in V(\Gamma_q)$
and $\beta\in N^3_{\Gamma_q}(\alpha)$.  Further $|A |= 2(q+1)^2$.

\item[(ii)]
{\small $ B=\displaystyle \bigcup_{c\in \mathbb{F}_q} N_{\Gamma_q}[(q,0,c)_{1}]\cup N_{\Gamma_q}[(q,q,0)_{1}]\cup\left( \bigcap_{c\in \mathbb{F}_q}N_{\Gamma_q}^2[(q,0,c)_{1}]\cap N_{\Gamma_q}^2[  (q,q,0)_{1}]\right)\cup N_{\Gamma_q}^2[(q,q,\xi)_{1}] ,$} where $\xi\in \mathbb{F}_q\setminus \{0\}$. Further $|B|=2(q^2+3q+1)$.

\item[(iii)]
    $$\begin{array}{lll}C&=&\displaystyle \bigcup_{x\in \mathbb{F}_q\cup \{q\}} N_{\Gamma_q}[(q,x,0)_{0}]\cup \left(\bigcap_{x\in \mathbb{F}_q\cup \{q\}} N_{\Gamma_q}^2[(q,x,0)_{0}] \right) \cup\bigcup_{x\in \mathbb{F}_q}N_{\Gamma_q}[(x,x,p(x))_{1}]\\[3ex]
   &&\displaystyle  \cup N_{\Gamma_q}[  (q,1,1)_{1}]\cup \left( \bigcap_{x\in \mathbb{F}_q}N_{\Gamma_q}^2[(x,x,p(x))_{1}]\cap N_{\Gamma_q}^2[  (q,1,1)_{1}]\right),\end{array}$$
where $q\ge 8$ is even and   $p(x)=1+x+x^2$ for all $x\in \mathbb{F}_q$. Further $|C|=2(q^2+4q+3)$.
\end{itemize}
\end{theorem}

The perfect dominating sets described in item $(ii)$ and $(iii)$ of Theorem \ref{perfect}
are depicted in Figure \ref{estruc} and in Figure \ref{estruc2}
respectively.


\begin{figure}[h]
\begin{center}
\psset{unit=0.7cm, linewidth=0.03cm}
   \begin{pspicture}(0,0)(14,10)

\cnode*(1,2){0.08}{000}\put(-0.6,2){\tiny$(0,0,0)_{0}$}\put(1.2,2){$\cdots$}
\cnode*(2,2){0.08}{00j}\put(1.5,1.6){\tiny$(0,0,j)_{0}$}
\cnode*(3,2){0.08}{100}\put(3,1.6){\tiny$(1,0,0)_{0}$}\put(3.2,2){$\cdots$}
\cnode*(4,2){0.08}{10j}\put(4.1,2.2){\tiny$(1,0,j)_{0}$}

\cnode*(7,2){0.08}{q00}\put(6.5,1.6){\tiny$(q,0,0)_{0}$}\put(7.2,2){$\cdots$}
\cnode*(8,2){0.08}{q0j}\put(8.,1.6){\tiny$(q,0,j)_{0}$}

\cnode(1,5){0.08}{Lq00}\put(-0.6,5.3){\tiny$(q,0,0)_{1}$}

\cnode(3,5){0.08}{Lq10}\put(1.5,5.3){\tiny$(q,0,1)_{1}$} \put(4,5){$\cdots\quad \cdots$}

\cnode(7,5){0.08}{Lqq0}\put(5.5,5.3){\tiny$(q,q,0)_{1}$}\put(7.2,5){$\cdots$}
\cnode(8,5){0.08}{Lqqj}\put(8,5.3){\tiny$(q,q,j)_{1}$}

\cnode(10.5,5){0.08}{L110}\put(9.6,4.6){\tiny$(\xi, 0,1)_{1}$}\put(11.,5){$\cdots$}
\cnode(12,5){0.08}{L11j}\put(11.,4.6){\tiny$(\xi, t,1)_{1}$}
\cnode(13.2,5){0.08}{L100}\put(12.5,4.6){\tiny$(\xi, 0,0)_{1}$}\put(13.5,5){$\cdots$}
\cnode(14.5,5){0.08}{L10j}\put(13.9,4.6){\tiny$(\xi, t,0)_{1}$}

\cnode(3,0){0.08}{L000}\put(2.5,-0.4){\tiny$(0,0,0)_{1}$}\put(4,0){$\cdots\quad \cdots$}
\cnode(7,0){0.08}{L0j0}\put(6.5,-0.4){\tiny$(0,0,j)_{1}$}

\cnode*(1,7.5){0.08}{qq0}\put(-0.6,7.5){\tiny$(q,q,0)_{0}$}
\cnode*(3,7.5){0.08}{qq1}\put(1.5,7.5){\tiny$(q,q,1)_{0}$} \put(4,7.5){$\cdots\quad \cdots$}

\cnode*(12,7.5){0.08}{q11}\put(10.5,7.5){\tiny$(q,\xi, 1)_{0}$}
\cnode*(7,7.5){0.08}{qqq}\put(7.2,7.5){\tiny$(q,q,q)_{0}$} \put(8.5,7.5){$\cdots\quad \cdots$}
\cnode*(14,7.5){0.08}{q10}\put(14,7.5){\tiny$(q,\xi, 0)_{0}$}

\cnode(5,9.5){0.08}{Lqqq}\put(4.5,9.7){\tiny$(q,q,q)_{1}$}
\cnode(11,9.5){0.08}{Lqq1}\put(10.5,9.7){\tiny$(q,q,\xi )_{1}$}

%

\ncline[]{-}    {   000 }   {   Lq00    }

\ncline[]{-}    {   100 }   {   Lq10    }

\ncline[]{-}    {   q00 }   {   Lqq0    }

\ncline[]{-}    {   00j }   {   Lq00    }

\ncline[]{-}    {   10j }   {   Lq10    }

\ncline[]{-}    {   q0j }   {   Lqq0    }

\ncline[]{-}    {   L000    }   {   000 }
\ncline[]{-}    {   L000    }   {   100 }
\ncline[]{-}    {   L000    }   {   q00 }

\ncline[]{-}    {   L0j0    }   {   00j }
\ncline[]{-}    {   L0j0    }   {   10j }
\ncline[]{-}    {   L0j0    }   {   q0j }

\ncline[]{-}    {   qq0 }   {   Lq00    }
\ncline[]{-}    {   q10 }   {   L100    }
\ncline[]{-}    {   qq1 }   {   Lq10    }

\ncline[]{-}    {   q11 }   {   L110    }
\ncline[]{-}    {   q11 }   {   L11j    }
\ncline[]{-}    {   qqq }   {   Lqqj    }
\ncline[]{-}    {   qqq }   {   Lqq0    }
\ncline[]{-}    {   q10 }   {   L10j    }

\ncline[]{-}    {   Lqqq    }   {   qq0 }
\ncline[]{-}    {   Lqqq    }   {   qq1 }
\ncline[]{-}    {   Lqqq    }   {   qqq }
\ncline[]{-}    {   Lqq1    }   {   qqq }
\ncline[]{-}    {   Lqq1    }   {   q11 }
\ncline[]{-}    {   Lqq1    }   {   q10 }

\end{pspicture}\end{center}
\caption{\label{estruc} Deleted subgraph in $(ii)$ of Theorem
\ref{q}.}
\end{figure}
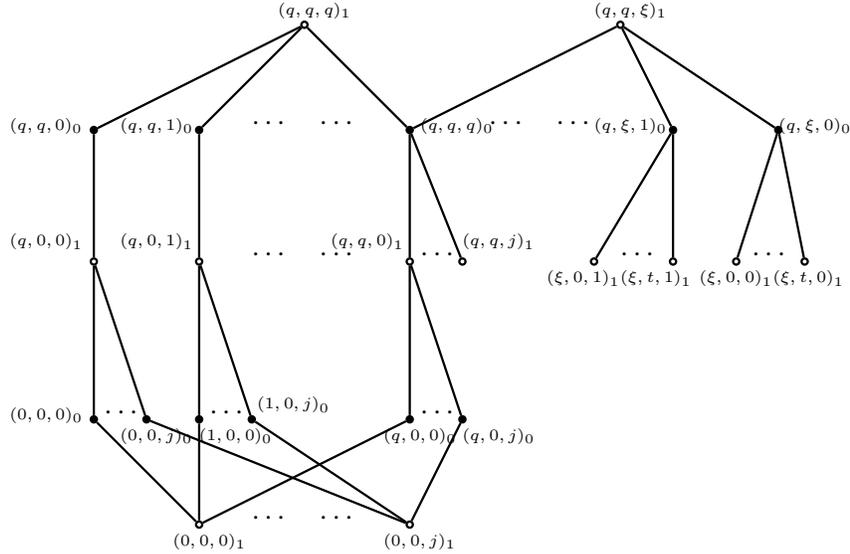

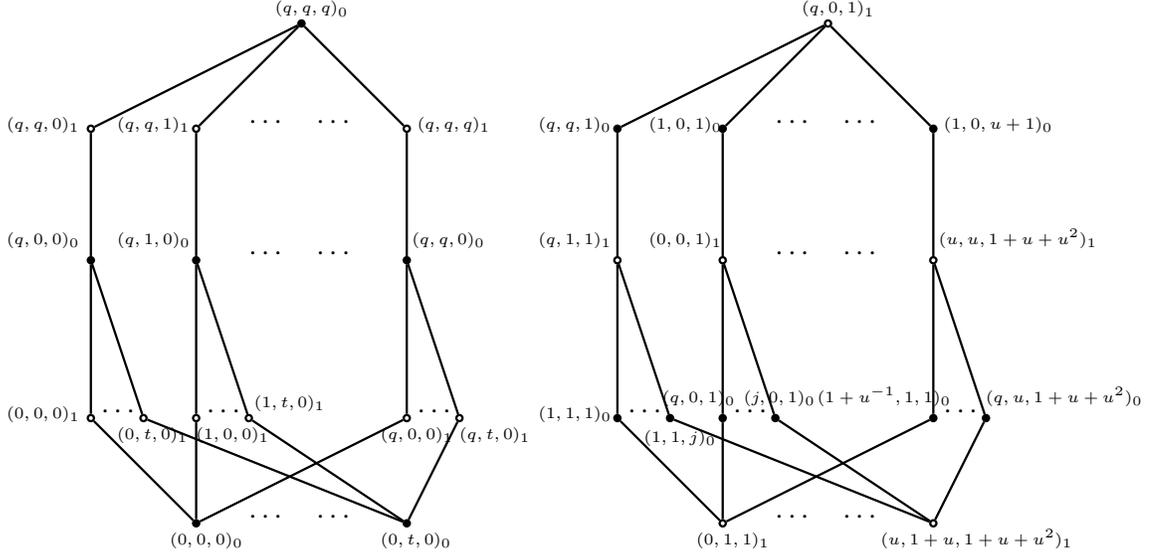
\begin{figure}[h]
\begin{center}
\psset{unit=0.7cm, linewidth=0.03cm}
   \begin{pspicture}(0,0)(17,10)

\cnode(1,2){0.08}{000}\put(-0.6,2){\tiny$(0,0,0)_{1}$}\put(1.2,2){$\cdots$}
\cnode(2,2){0.08}{00j}\put(1.5,1.6){\tiny$(0,t,0)_{1}$}
\cnode(3,2){0.08}{100}\put(3,1.6){\tiny$(1,0,0)_{1}$}\put(3.2,2){$\cdots$}
\cnode(4,2){0.08}{10j}\put(4.1,2.2){\tiny$(1,t,0)_{1}$}

\cnode(7,2){0.08}{q00}\put(6.5,1.6){\tiny$(q,0,0)_{1}$}\put(7.2,2){$\cdots$}
\cnode(8,2){0.08}{q0j}\put(8.,1.6){\tiny$(q,t,0)_{1}$}

\cnode*(1,5){0.08}{Lq00}\put(-0.6,5.3){\tiny$(q,0,0)_{0}$}

\cnode*(3,5){0.08}{Lq10}\put(1.5,5.3){\tiny$(q,1,0)_{0}$}
\put(4,5){$\cdots\quad \cdots$}

\cnode*(7,5){0.08}{Lqq0}\put(7.1,5.3){\tiny$(q,q,0)_{0}$}

\cnode*(3,0){0.08}{L000}\put(2.5,-0.4){\tiny$(0,0,0)_{0}$}\put(4,0){$\cdots\quad
\cdots$} \cnode*(7,0){0.08}{L0j0}\put(6.5,-0.4){\tiny$(0,t,0)_{0}$}

\cnode(1,7.5){0.08}{qq0}\put(-0.6,7.5){\tiny$(q,q,0)_{1}$}
\cnode(3,7.5){0.08}{qq1}\put(1.5,7.5){\tiny$(q,q,1)_{1}$}
\put(4,7.5){$\cdots\quad \cdots$}

\cnode(7,7.5){0.08}{qqq}\put(7.2,7.5){\tiny$(q,q,q)_{1}$}

\cnode*(5,9.5){0.08}{Lqqq}\put(4.5,9.7){\tiny$(q,q,q)_{0}$}

%

\ncline[]{-}    {   000 }   {   Lq00    }

\ncline[]{-}    {   100 }   {   Lq10    }

\ncline[]{-}    {   q00 }   {   Lqq0    }

\ncline[]{-}    {   00j }   {   Lq00    }

\ncline[]{-}    {   10j }   {   Lq10    }

\ncline[]{-}    {   q0j }   {   Lqq0    }

\ncline[]{-}    {   L000    }   {   000 } \ncline[]{-}    {   L000
}   {   100 } \ncline[]{-}    {   L000    }   {   q00 }

\ncline[]{-}    {   L0j0    }   {   00j } \ncline[]{-}    {   L0j0
}   {   10j } \ncline[]{-}    {   L0j0    }   {   q0j }

\ncline[]{-}    {   qq0 }   {   Lq00    }

\ncline[]{-}    {   qq1 }   {   Lq10    }

\ncline[]{-}    {   qqq }   {   Lqq0    } \ncline[]{-}    {   q10
}   {   L10j    }

\ncline[]{-}    {   Lqqq    }   {   qq0 } \ncline[]{-}    {   Lqqq
}   {   qq1 } \ncline[]{-}    {   Lqqq    }   {   qqq }

\cnode*(11,2){0.08}{q01}\put(9.5,2){\tiny$(1,1,1)_{0}$}
\put(11.2,2){$\cdots$}

\cnode*(12,2){0.08}{j01}\put(11.5,1.55){\tiny$(1,1,j)_{0}$}

\cnode*(13,2){0.08}{111}\put(11.85 ,2.3) {\tiny$(q,0,1)_{0}$}
\put(13.2,2){$\cdots$}
\cnode*(14,2){0.08}{11j}\put(13.4,2.3){\tiny$(j,0,1)_{0}$}

\cnode*(17,2){0.08}{v}\put(14.8,2.3){\tiny$(1+u^{-1},1,1)_{0}$}\put(17.2,2){$\cdots$}

\cnode*(18,2){0.08}{q0j}\put(18.,2.3){\tiny$(q,u,1+u+u^2)_{0}$}

\cnode(11,5){0.08}{L010}\put(9.5,5.3){\tiny$(q,1,1)_{1}$}

\cnode(13,5){0.08}{Lq11}\put(11.6,5.3){\tiny$(0,0,1)_{1}$}\put(14,5){$\cdots\quad
\cdots$}

\cnode(17,5){0.08}{Lu1}\put(17.1,5.3){\tiny$(u,u,1+u+u^2)_{1}$}

\cnode(13,0){0.08}{L011}\put(12.5,-0.4){\tiny$(0,1,1)_{1}$}\put(14,0){$\cdots\quad
\cdots$}
\cnode(17,0){0.08}{Lv}\put(16,-0.4){\tiny$(u,1+u,1+u+u^2)_{1}$}

\cnode*(11,7.5){0.08}{101}\put(9.5,7.5){\tiny$(q,q,1)_{0}$}
\cnode*(13,7.5){0.08}{qq1}\put(11.6,7.5){\tiny$(1,0,1)_{0}$}
\put(14,7.5){$\cdots\quad \cdots$}

\cnode*(17,7.5){0.08}{qqq}\put(17.2,7.5){\tiny$(1,0,u+1)_{0}$}

\cnode(15,9.5){0.08}{Bq10}\put(14.5,9.7){\tiny$(q,0,1)_{1}$}

%
\ncline[]{-}    {   q01 }   {   L011    } \ncline[]{-}    {   v
}   {   L011    } \ncline[]{-}    {   111 }   {   L011    }
\ncline[]{-}    {   q01 }   {   L010    }



\ncline[]{-}    {   j01 }   {   L010    } \ncline[]{-}    {   111
}   {   Lq11    } \ncline[]{-}    {   11j }   {   Lq11    }

\ncline[]{-}    {   q0j }   {   Lu1 } \ncline[]{-}    {   v   }
{   Lu1 }


\ncline[]{-}    {   Lv  }   {   j01 } \ncline[]{-}    {   Lv  }
{   11j } \ncline[]{-}    {   Lv  }   {   q0j }

\ncline[]{-}    {   101 }   {   L010    }

\ncline[]{-}    {   qq1 }   {   Lq11    }

\ncline[]{-}    {   qqq }   {   Lu1 } \ncline[]{-}    {   q10 }
{   L10j    }

\ncline[]{-}    {   Bq10    }   {   101 } \ncline[]{-}    {   Bq10
}   {   qq1 } \ncline[]{-}    {   Bq10    }   {   qqq }
\end{pspicture}\end{center}
\caption{\label{estruc2} Deleted subgraph in $(iii)$ of Theorem
\ref{q}.}
\end{figure}
\begin{remark}\label{rem2}
\begin{enumerate}
\item[(a)] Suppose $q=2$.  A cycle of
length 8 is obtained by eliminating from  the bipartite graph  $\Gamma_2$ the
vertices  of the set $B$ from
Theorem \ref{perfect}  $(ii)$. And   the $(3,8)$-cage can be partitioned into  the
two induced subgraphs shown
in Figure \ref{estruc2}.

\item[(b)]
For $q=4$,   $p(x)=1+x+x^2\in \{0,1\}$ for all $x\in \mathbb{F}_4$. Taking $\xi\in \mathbb{F}_4\setminus\{0,1\}$,  we can find, for the $(5,8)$-cage,  the following   perfect dominating set similar to $C$ of Theorem \ref{perfect}:
$$\begin{array}{lll}C'&=&\displaystyle \bigcup_{x\in \mathbb{F}_4} N_{\Gamma_4}[(4,x,\xi)_{0}]\cup N_{\Gamma_4}[(4,4,0)_{0}]\cup\left(\bigcap_{x\in \mathbb{F}_4\cup \{4\}} N_{\Gamma_4}^2[(4,x,\xi)_{0}]\cap N_{\Gamma_4}[(4,4,0)_{0}] \right)\\[3ex]  &&  \bigcup_{x\in \mathbb{F}_4}N_{\Gamma_4}[(x,x,p(x))_{1}]
   \displaystyle  \cup N_{\Gamma_4}[    (4,1,1)_{1} ]\cup \left( \bigcap_{x\in \mathbb{F}_4}N_{\Gamma_4}^2[(x,x,p(x))_{1}]\cap N_{\Gamma_4}^2[    (4,1,1)_{1} ]\right).\end{array}$$
   \end{enumerate}
\end{remark}

The following result is  an  immediate 
consequence of Theorem \ref{perfect} and Remark \ref{rem2} (b).

\begin{theorem}\label{q}Let $q\ge 2$ be a  prime power  and  $\Gamma_q=\Gamma_q[V_{0},V_{1}]$ the $(q+1, 8)$-cage constructed in Theorem \ref{main8}.
Removing from $\Gamma_q$ the perfect dominating sets from Theorem \ref{perfect},
$q$-regular graphs of girth 8 are obtained of orders $2q(q^2-1)$, $2q(q^2-2)$ for any
prime power $q$ or of order $2(q^3-3q-2)$ for even prime powers $q \ge 4$.
\end{theorem}
G\'acs and H\'eger   \cite{GH08} obtain $(q,8)$-bipartite graphs  on $2q(q^2-2)$
vertices if $q$ is odd, or on $2(q^3-3q-2)$ vertices if $q$ is
even,  using a classical generalized quadrangle $GQ$ and assuming
that $GQ$  has a substructure called regular point-pair $(u, v)$. Note
that in Theorem \ref{q} we obtain explicitly $(q,8)$-bipartite graphs on the same cardinality 
using Definition \ref{gammaq} without assuming anything. 
Moreover, using classical GQ, Beukemann and   Metsch  \cite{BM10} prove that the
cardinality of a perfect dominating set $ B $ is at most $|B|\le 2(2q2+2q)$ and if $q$ is even
$|B|\le 2(2q2+q+1)$. 
And   $(k,8)$-regular balanced
bipartite graphs for all prime  powers  $q$  such that $3\le
k\le q$ of order $2k ( q^2-1)$ have been obtained  as   subgraphs
of the incidence graph of a generalized quadrangle \cite{ABH10}.
This result has been improved by constructing $(k,8)$-regular
balanced bipartite graphs of order $2q (kq -1)$  in  \cite{B09}.

 To finish we improve these results for the case $k=q-1$.
 \begin{definition}\label{GrGq}
Let $q\ge 4$ be a  prime power  and  $G_q$ the $q$-regular graph of girth 8 constructed in Theorem \ref{q} on $2q(q^2-2)$ vertices choosing $\xi\in \mathbb{F}_q\setminus \{0,1\}$.
\end{definition}

  Given a subset of vertices $S\subset V(G)$ we denote by $N_G(S)=\cup_{s\in S}N_G(s)$.

\begin{theorem}\label{q-1}
Let $q\ge 4$ be a  prime power  and   $G_q$ the graph from Definition~\ref{GrGq}.
Define $R=N_{G_q}(\{(q,y,z)_0: y,z\in \mathbb{F}_q, y\not=0,1,\xi\})\cap N^5_{G_q}((q,1,0)_0)$.
The set
 $$ S:= \bigcup_{z\in \mathbb{F}_q} N_{G_q}[(q,1,z)_0]\cup   N_{G_q}[R].$$
is perfect dominating in $G_q$.
Furthermore, $G_q-S$ a $(q-1)$-regular graph of girth 8 of order $2q(q-1)^2$.
\end{theorem}


\section{$(q+1,8)$-cages}
In order to prove Theorem \ref{main8} we will first define two \emph{auxiliary} graphs $H_q$ and $B_q$ (c.f. Definitions \ref{Hq}, \ref{Bq}, which were inspired by the construction of Lazebnik and Ustimenko \cite{LU95}
    of a family of $q$-regular graphs $D(n,q)$, $n\ge 2$ and $q$ a prime power, of order $2q^n$ and girth at least $n+5$ for $n$ odd (and at least $n+4$ for $n$ even). In particular when $n=3$ the graph $D(3,q)$ has $2q^3$ vertices and girth 8.
    In what follows we construct another  $q$-regular bipartite graph  $H_q$  of girth 8
    as a first step  to achieve our goal. It can be checked that $D(3,q)$ and  $H_q$  are not isomorphic for
    $q\ge 3$.

\begin{definition}\label{Hq}
Let  $\mathbb{F}_q$ be a finite field with $q\ge 2$.
Let $H_{q }=H_q[U_0,U_1]$ be a bipartite graph with vertex set $U_r=\mathbb{F}_q\times \mathbb{F}_q\times \mathbb{F}_q$, $r=0,1$; and edge set $E(H_q)$ defined as follows:
$$
\mbox{For all }   a, b,c\in \mathbb{F}_q  : N_{H_q}((a,b,c)_{1} )=
    \{(x,~ax+b,~a^2x +c)_{0}: x \in \mathbb{F}_q \} .
 $$
\end{definition}
\begin{lemma}\label{claim01}
Let $H_q$ be the graph from Definition \ref{Hq}.
For any given $a\in \mathbb{F}_q $,  the vertices in the set
$\{(a,b,c)_{1}:  b,c \in \mathbb{F}_q \}$ are mutually at distance
at least four.
Also, for any given $x\in \mathbb{F}_q $, the vertices in the set
$\{(x,y, z)_{0}:  y,z  \in \mathbb{F}_q \}$ are mutually at distance at least four.
\end{lemma}

\begin{proof}
Suppose that
there exists in $H_q$ a path of length two $(a,b,c )_{1}  (j,y,z)_{0}  (a,b',c' )_{1}$ with $b\ne b'$ or $c\ne c'$. Then $y=aj+b=aj+b'$ and $z=a^2j+c=a^2j+c'$. Hence
$b=b'$ and $c=c'$   which is a contradiction. Similarly suppose
that there exists a path $(x,y, z)_{0}  (a,b,c )_{1}  (x,y', z')_{0}$ with $y\ne y'$ or $z\ne z'$.
Reasoning similarly, we obtain   $y=ax+b=y'$, and   and $z=a^2x+c=z'$
which is a contradiction.
\end{proof}

\begin{proposition}{\label{bip8}}
The graph $H_q$ from Definition \ref{Hq} is a $q$-regular bipartite of girth $8$ and order $2q^3$.
\end{proposition}

\begin{proof}For $q=2$ it can be
checked that $H_2$ consists of two disjoint cycles of length 8. Thus we assume that $q\ge 3$.
Clearly $H_q $ has order $2q^3$  and every vertex of $U_{1}$ has degree $q$. Let $(x,y,z)_{0}\in U_{0}$.
 By definition of $H_q$,
 \begin{equation}\label{cadual}
 N_{H_q}((x,y,z)_{0})= \left\{
 (a,~y-ax,~z-a^2x))_{1} : a\in \mathbb{F}_q  \right\}.
  \end{equation}
  Hence every vertex of $U_{0}$ has also degree $q$ and $H_q$ is $q$-regular.
Next, let us prove that $H_q$ has no cycles of length less than 8.  Otherwise suppose that there exists in $H_q$ a cycle
$$C_{2t+2}=(a_{0},b_{0},c_{0}  )_{1} (x_{0}, y_{0},z_{0} )_{0} (a_{1},b_{1},c_{1} )_{1} \cdots  (x_{t},y_{t},z_{t} )_{0} (a_{0},b_{0},c_{0}  )_{1} $$ of length $2t+2$ with $t\in \{1,2\}$.
By Claim 0, $a_k\ne a_{k+1}$ and $x_k\ne x_{k+1}$ (subscripts being
  taken modulo $t+1$).
Then
  $$ \begin{array}{lll}y_{k} =a_kx_k+b_{k}  &=&a_{k+1}x_k+b_{k+1}, \quad    \, k=0, \ldots, t , \\
  z_{k} =a_k^2x_k+c_{k} &=&a_{k+1}^2x_k+c_{k+1}, \quad    \, k=0, \ldots, t ,
 \end{array} $$
 subscripts $k$ being
  taken modulo $t+1$. Summing all these equalities we get
\begin{equation}\label{suma}\begin{array}{lll}
\displaystyle \sum_{k=0}^{t-1}(a_k -a_{k+1} )x_k&=&(a_{1} -a_{t} )x_{t}, \quad t =1,2.\\[2ex]
\displaystyle \sum_{k=0}^{t-1}(a_k^2-a_{k+1}^2)x_k&=&(a_{1}^2-a_{t}^2)x_{t}, \quad t=1,2.\end{array}\end{equation}
 If $t=1$, then     (\ref{suma}) leads to $(a_{1}-a_{1})(x_{1}-x_{0})=0~$. Then $a_{1}=a_{0}$ or $x_{1}=x_{0}$ which is a contradiction by Claim 0. This means that $H_q$ has no squares so that
  we may assume that $t=2$. The coefficient matrix of (\ref{suma}) has a Vandermonde determinant, i.e.:
$$
\left|\begin{array}{cc } a_{1}-a_{0}  & a_{0}-a_2  \\
 a_{1}^2-a_{0}^2 & a_{0}^2-a_2^2  \\
 \end{array}\right|=
 \left|\begin{array}{ccc}  1  & 1 &  1 \\
a_{1}   & a_{0}  &  a_{2}  \\
a_{1}^2 &a_{0}^2& a_{2}^2\\
\end{array}\right|
= \displaystyle \prod_{0\le k<j\le 2}(a_j-a_k)
$$
This determinant is different from zero
  because by Claim 0, $a_{k+1}\ne a_k$ (the subscripts being taken modulo $3$).   Using Cramer's rule to solve it
 we obtain $x_{1}=x_{0}=x_2$ which is a
contradiction with Claim 0.

 Hence, $H_q$ has girth at least 8. Furthermore, when $q\ge 3$ the
minimum number of vertices of a $q$-regular bipartite graph of
girth greater than 8 must be greater than $2q^3$.
 Thus we conclude
that the girth of $H_q$  is exactly 8.
\end{proof}

\begin{definition}\label{Bq}
Let $B_q$ be a bipartite graph with vertex set $V(B_q)=(\mathbb{F}_q^3,~\mathbb{F}_q^3)$,
and edge set $E(B_q)$ defined as follows:
$$
\mbox{For all }   a, b,c\in \mathbb{F}_q  : N_{B_q}((a,b,c)_{1} )=
    \{(j,~aj+b,~a^2j+2ab+c)_{0}: j \in \mathbb{F}_q \} .
$$
\end{definition}

\begin{lemma}\label{HqisoBq}
The graph $B_q$ is $q$-regular, has girth 8, order $2q^3$ and
is isomorphic to the graphs $H_q$.
\end{lemma}

\begin{proof}
Let $H_q$ be the bipartite graph from Definition \ref{Hq}. Since the map $\sigma: B_q \to H_q$ defined by
$\sigma((a,b,c)_{1})=(a,b,2ab+c)_{1}$ and $\sigma((x,y,z)_{0})=(x,y,z)_{0}$ is an isomorphism,
the result holds.
\end{proof}

\begin{proof} {\bf of Theorem \ref{main8}:}
We will (re)-construct the graph $\Gamma_q$ from the graph $B_q$ adding some new vertices and edges.
Reasoning as in Lemma \ref{claim01} the following claim follows:

 \emph{Claim 1: For any given $a\in \mathbb{F}_q $,  the vertices of the set
$\{(a,b,c)_{1}:  b,c \in \mathbb{F}_q \}$ are mutually at distance
at least four in $B_q$. Also for any given $
x\in \mathbb{F}_q $, the vertices of set $\{(x,y, z)_{0}:  y,z  \in \mathbb{F}_q \}$ are mutually at distance
at least four in $B_q$.}

As a consequence of Claim 1 we obtain the following claim.

\noindent\emph{Claim 2: For all $x,y\in\mathbb{F}_q $, the $q$
vertices of the set $ \{(x,y,j)_{0}: j\in
\mathbb{F}_q \}$ are mutually at
distance at least 6 in $B_q$.}

{\it Proof:}
By Claim 1,   the
  $q$ vertices $\{(x,y,j )_{0}:j\in
\mathbb{F}_q\}$   are mutually at distance at least 4. Suppose
that $B_q$ contains the following path of length four:
$$(x,y,j)_{0}  ~(a,b,c )_{1}~  (x',y',j')_{0} ~ (a',b',c' )_{1}  ~(x,y,j'')_{0}, \mbox{ for some } j''\ne j.$$
Then $y=  ax +b=  a' x+b'$ and $y'=   ax'+b  =  a'x'+b' $.
It follows that $  (a-a')(x-x') =0$, which is a contradiction because    $a\ne a'$ and $x\ne x'$ by Claim 1.  \quad $\Box$

Let $B'_{q}=B'_q[V_{0},V_{1}']$ be the bipartite graph obtained from $B_q=B_q[V_{0},V_{1}]$ by adding $q^2$ new vertices to $V_{1}$ labeled $(q,b,c)_{1}$,  $b,c\in \mathbb{F}_q$ (i.e., $V_{1}'=V_1\cup  \{(q,b,c)_{1}: b,c\in \mathbb{F}_q\}$), and new edges $N_{B'_{q}}((q,b,c)_{1})= \{(c, b ,j )_{0}: j \in \mathbb{F}_q \}$ (see Figure \ref{spanning}).
Then $B'_{q}$ has $|V'_{1}|+|V_{0}|=2q^3+q^2$ vertices such that every vertex of $V_{0}$ has degree $q+1$ and every vertex of $V_{1}'$ has still degree $q$. Note that the girth of $B'_{q}$ is 8 by  Claim 2. Further, Claim 1    partially holds in $B'_{q}$. We write this fact in the following claim.

\noindent\emph{Claim 3: For any given $a\in \mathbb{F}_q\cup\{q\} $,  the vertices of the set
$\{(a,b,c )_{1}:  b,c \in \mathbb{F}_q \}$ are mutually at distance
at least four in $B'_{q}$.}

  \bigskip

\noindent\emph{Claim 4: For all $a\in\mathbb{F}_q\cup \{q\} $ and for all $c\in\mathbb{F}_q $, the $q$
vertices of the set $ \{(a,t,c)_{1}: t\in
\mathbb{F}_q \}$ are mutually at
distance at least 6 in $B'_{q}$.}

{\it Proof:}
By Claim 3, for all $a\in\mathbb{F}_q \cup \{q\}  $ the $q$ vertices of
 $ \{(a,t, c)_{1}: t\in
\mathbb{F}_q\}$  are  mutually at distance at least 4   in  $B'_{q}$. Suppose
that there exists in $B'_{q}$ the following path of length four:
$$(a,t, c)_{1}  ~(x,y,z)_{0} ~ (a',t',c')_{1} ~ (x',y',z')_{0} ~ (a,t'',c)_{1}, \mbox{ for some } t''\ne t.$$

 \noindent If $a=q$, then $x=x'=c$, $y=t$, $y'=t''$ and $a'\ne q$ by Claim 3.  Then $y=a' x+t'=a' x'+t'=y'$ yielding that $t=t''$ which is a contradiction. Therefore
 $a\ne q$.   If $a'=q$, then $x=x'=c'$ and $y=y'=t'$. Thus $y=a  x+t=ax'+ t''=y'$ yielding that $t=t''$ which is a contradiction. Hence we may assume that $a'\ne q$ and $a\ne a'$ by Claim 3. In this case we have:
  $$
    \begin{array}{ll}
               y=ax+ t &= a'x+t' ;\\
           y'= ax'+t''      &=  a'x'+t'  ;\end{array}
    \begin{array}{ll}
              z=  a^2x +    2a t+c & =  a'^2x +  2a't' +c';\\
               z'=   a^2x' +    2 at''+c & =  a'^2x' +  2a't'+c'.
              \end{array}
              $$
 Hence
 \begin{eqnarray}
 \label{dosa} (a-a')(x-x')&=t''-t;\\
 \label{dosb} (a^2-a'^2)(x-x')&= 2a(t''- t).
\end{eqnarray}
If $q$ is even, (\ref{dosb}) leads to  $x=x'$ and (\ref{dosa}) leads to $t''=t$   which is a contradiction with our assumption.
Thus assume $q$ odd. If $a+a'=0$, then  (\ref{dosb}) gives
              $2a(t''-t)=0$, so that $ a=0$ yielding that $a'=0$ (because $a+a'=0$)
              which is again a contradiction.
If $a+a'\ne 0$, multiplying  equation  (\ref{dosa}) by $a+a'$  and resting both equations we obtain  $(2a-(a+a'))(t''-t)=0$. Then   $a=a'$ because $t''\ne t$, which is a
              contradiction to Claim 3.
  Therefore, Claim 4 holds.  \quad $\Box$

\bigskip

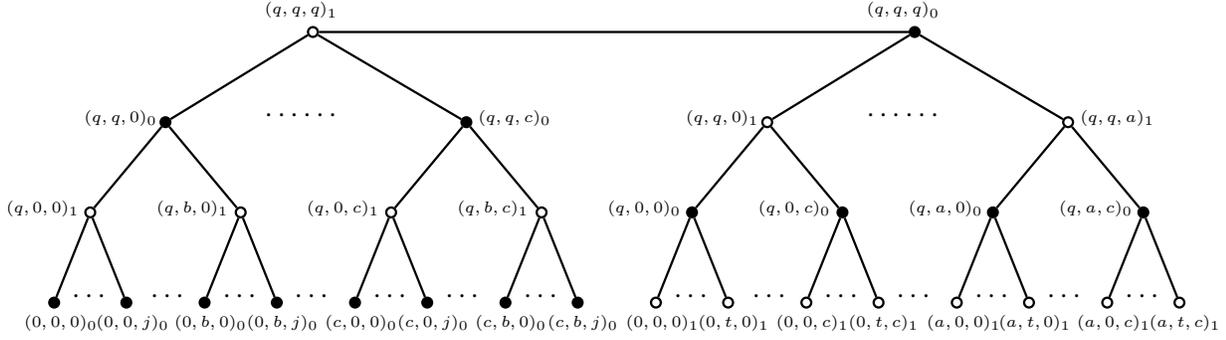
\begin{figure}
\begin{center}
\psset{unit=0.8mm, linewidth=0.03cm}
   \begin{pspicture}(75,-2)(100,47)

\cnode*(-5,0){1}{000}\put(-10,-4){\tiny$(0,0,0)_0$}
\put(-2,0){$\cdots$}
\cnode*(7,0){1}{00j}   \put(2,-4){\tiny$(0,0,j)_0$}
\put(11,0){$\cdots$}
\cnode*(20,0){1}{0b0} \put( 15,-4){\tiny$(0,b,0)_0$}
\put(23,0){$\cdots$}
\cnode*(32,0){1}{0bj}  \put(27,-4){\tiny$(0,b,j)_0$}
\put(35,0){$\cdots$}
 \cnode*(45,0){1}{c00}\put(40,-4){\tiny$(c,0,0)_0$}
 \put(48,0){$\cdots$}
\cnode*(57,0){1}{c0j}   \put(52,-4){\tiny$(c,0,j)_0$}
\put(60,0){$\cdots$}
\cnode*(70,0){1}{cb0} \put(65,-4){\tiny$(c,b,0)_0$}
\put(73,0){$\cdots$}
\cnode*(82,0){1}{cbj}  \put(77,-4){\tiny$(c,b,j)_0$}


\cnode(95,0){1}{L000}\put(90,-4){\tiny$(0,0,0)_1$}
\put(98,0){$\cdots$}
\cnode(107,0){1}{L0t0}   \put(102,-4){\tiny$(0,t,0)_1$}
\put(111,0){$\cdots$}
\cnode(120,0){1}{L00c} \put( 115,-4){\tiny$(0,0,c)_1$}
\put(123,0){$\cdots$}
\cnode(132,0){1}{L0tc}  \put(127,-4){\tiny$(0,t,c)_1$}
\put(135,0){$\cdots$}
 \cnode(145,0){1}{La00}\put(140,-4){\tiny$(a,0,0)_1$}
 \put(148,0){$\cdots$}
\cnode(157,0){1}{Lat0}   \put(152,-4){\tiny$(a,t,0)_1$}
\put(160,0){$\cdots$}
\cnode(170,0){1}{La0c} \put( 165,-4){\tiny$(a,0,c)_1$}
\put(173,0){$\cdots$}
\cnode(182,0){1}{Latc}  \put(177,-4){\tiny$(a,t,c)_1$}

%
%
%
 \cnode(1,15){1}{Lq00}   \put(-13,15){\tiny$(q,0,0)_1$}
 \cnode(26,15){1}{Lqb0}   \put(12,15){\tiny$(q,b,0)_1$}
 \cnode(51,15){1}{Lq0c}   \put(37,15){\tiny$(q,0,c)_1$}
 \cnode(76,15){1}{Lqbc}   \put(62,15){\tiny$(q,b,c)_1$}
\cnode*(101,15){1}{q00}   \put(87,15){\tiny$(q,0,0)_0$}
 \cnode*(126,15){1}{q0c}   \put(112,15){\tiny$(q,0,c)_0$}
 \cnode*(151,15){1}{qa0}   \put(137,15){\tiny$(q,a,0)_0$}
 \cnode*(176,15){1}{qac}   \put(162,15){\tiny$(q,a,c)_0$}

\cnode*(13.5,30){1}{qq0}   \put(0,30){\tiny$(q,q,0)_0$}
\put(30,30){$\cdots\cdots$}
\cnode*(63.5,30){1}{qqc}   \put(65.5,30){\tiny$(q,q,c)_0$}
\cnode(113.5,30){1}{Lqq0}   \put(100,30){\tiny$(q,q,0)_1$}
\put(130,30){$\cdots\cdots$}
\cnode(163.5,30){1}{Lqqa}   \put(165.5,30){\tiny$(q,q,a)_1$}
\cnode(38,45){1}{Lqqq}   \put(30,48){\tiny$(q,q,q)_1$}
\cnode*(138,45){1}{qqq}   \put(130,48){\tiny$(q,q,q)_0$}

\ncline[]{-}{Lq00}{000}
 \ncline[]{-}{Lq00}{00j}
 \ncline[]{-}{Lqb0}{0b0}
 \ncline[]{-}{Lqb0}{0bj}
 \ncline[]{-}{Lq0c}{c00}
 \ncline[]{-}{Lq0c}{c0j}
 \ncline[]{-}{Lqbc}{cb0}
 \ncline[]{-}{Lqbc}{cbj}
 \ncline[]{-}{q00}{L000}
  \ncline[]{-}{q00}{L0t0}
  \ncline[]{-}{q0c}{L00c}
  \ncline[]{-}{q0c}{L0tc}
  \ncline[]{-}{qa0}{La00}
  \ncline[]{-}{qa0}{Lat0}
  \ncline[]{-}{qac}{La0c}
  \ncline[]{-}{qac}{Latc}

  \ncline[]{-}{qq0}{Lq00}
  \ncline[]{-}{qq0}{Lqb0}
  \ncline[]{-}{qqc}{Lq0c}
  \ncline[]{-}{qqc}{Lqbc}
 \ncline[]{-}{Lqq0}{q00}
  \ncline[]{-}{Lqq0}{q0c}
 \ncline[]{-}{Lqqa}{qa0}
  \ncline[]{-}{Lqqa}{qac}
   \ncline[]{-}{Lqqq}{qq0}
 \ncline[]{-}{Lqqq}{qqc}
\ncline[]{-}{Lqqq}{qqq}
\ncline[]{-}{qqq}{Lqq0}
\ncline[]{-}{qqq}{Lqqa}

      \end{pspicture}
\caption{Spanning tree of $\Gamma_q$.   \label{spanning}}
\end{center}
\end{figure}
Let $B''_{q}=B''_{q}[V'_{0},V_{1}']$ be the graph obtained from $B'_{q}=B'_{q}[V_0,V'_1]$ by adding $q^2+q$ new vertices to $V_{0}$ labeled $(q,a,c)_{0}$, $a\in\mathbb{F}_q\cup \{q\}$, $c\in \mathbb{F}_q$,   and new edges $N_{B''_{q}}((q,a,c)_{0})= \{(a, t ,c )_{1}: t \in \mathbb{F}_q \}$ (see Figure \ref{spanning}).
Then $B''_{q}$ has $|V'_{1}|+|V_{0}'|=2q^3+2q^2+q$ vertices such that every vertex  has degree $q+1$ except the new added vertices which have degree $q$. Moreover the girth of $B''_{q}$ is 8 by  Claim 4.

\emph{Claim 5: For all $a\in\mathbb{F}_q\cup \{q\} $, the $q$
vertices of the set $ \{(q,a,j)_{0}: j\in
\mathbb{F}_q \}$ are mutually at
distance at least 6 in $B''_{q}$.}

 {\it Proof:}
Clearly   these $q$ vertices are  mutually at distance at least 4   in  $B''_{q}$. Suppose
that there exists in $B''_{q}$ the following path of length four:
$$(q,a, j)_{0}~ (a,b,j)_{1}~ (x,y,z)_{0}~ (a,b',j')_{1}~ (q,a,j')_{0}, \mbox{ for some } j'\ne j.$$
If $a=q$ then $x=j=j'$ which is a contradiction.  Therefore $a\ne q$. In this case $y=ax+b=ax+b'$ which implies that $b=b'$. Hence $z=a^2x+2ab+j=a^2x+2ab'+j'$ yielding that $j=j'$ which is again a contradiction. \quad $\Box$

Let $B'''_{q}=B'''_{q}[V_{0}',V''_{1}]$ be the graph obtained from $B''_{q}$ by adding $q+1$ new vertices to $V_{1}'$ labeled $(q,q,a)_{1}$,  $a\in \mathbb{F}_q\cup \{q\}$,    and new edges $N_{B'''_{q}}(q,q,a)_{1}= \{(q,  a,c )_{0}: c\in \mathbb{F}_q \}$, see Figure \ref{spanning}.
Then $B'''_{q}$ has $|V''_{1}|+|V_{0}'|=2q^3+2q^2+2q+1$ vertices such that every vertex  has degree $q+1$ except the new added vertices which have degree $q$. Moreover the girth of $B'''_{q}$ is 8 by  Claim 5 and clearly   these $q+1$ new vertices are mutually at distance 6. Finally,    the   $(q+1,8)$-cage $\Gamma_q$ is obtained by adding to $B'''_{q}$ another new vertex labeled $(q,q,q)_{0}$ and edges  $N_{\Gamma_q}((q,q,q)_{0})= \{(q,q,i)_{1}: i\in \mathbb{F}_q\cup \{q\}\}$.
\end{proof}


\subsection{Small $(q,8)$-graphs}

\begin{proof}{\bf of Theorem \ref{perfect}:}
$(i)$  Let $A=N_{\Gamma_q}^2[\alpha]\cup N_{\Gamma_q}^2[\beta] $ where $\alpha,\beta \in V(\Gamma_q)$
and $\beta\in N^3_{\Gamma_q}(\alpha)$.
Since the girth of $\Gamma_q$ is 8 there is a unique path of length three joining $\alpha$ and $\beta$.  Hence $|N_{\Gamma_q}^2[\alpha]\cap
N_{\Gamma_q}^2[\beta]|=|N_{\Gamma_q}(\alpha)\cap N^2_{\Gamma_q}(\beta)|+|N_{\Gamma_q}(\beta)\cap N^2_{\Gamma_q}(\alpha)|=2$ yielding that  $|A |=|N_{\Gamma_q}^2[\alpha]\cup
N_{\Gamma_q}^2[\beta]|=2(1+q+1+(q+1)q)-2=2(q+1)^2$.

Also since $\alpha$ and $\beta$ are at distance three, $  N_{\Gamma_q}^i(\alpha)$ and  $  N_{\Gamma_q}^i(\beta)$ are contained in different partite sets for all $i=0,1,2,3,4$.
Moreover, since the diameter
of $\Gamma_q$ is four,
$V(\Gamma_q)= N_{\Gamma_q}^2[\alpha] \cup N_{\Gamma_q}^3(\alpha)
\cup N_{\Gamma_q}^4(\alpha) =N_{\Gamma_q}^2[\beta] \cup N_{\Gamma_q}^3(\beta)
\cup N_{\Gamma_q}^4(\beta)$.   Hence if $v\not \in A$ then $v\in N_{\Gamma_q}^3(\alpha) \cup N_{\Gamma_q}^4(\alpha)$. If $v\in  N_{\Gamma_q}^3(\alpha)$ then $|N_{\Gamma_q}(v)\cap A|=|N_{\Gamma_q}(v)\cap N^2_{\Gamma_q}(\alpha)|=|N_{\Gamma_q}(v)\cap N^2_{\Gamma_q}[\alpha]|=1$ because the girth is 8.
If $v\in N_{\Gamma_q}^4(\alpha)$  then $|N_{\Gamma_q}(v)\cap A|=|N_{\Gamma_q}(v)\cap N^2_{\Gamma_q}(\beta)|=1$. Therefore $A$ is a perfect dominating set of $\Gamma_q$.

 $(ii) $ From
Theorem \ref{main8}, it follows that $ \bigcap_{c\in \mathbb{F}_q}N_{\Gamma_q}^2[(q,0,c)_{1}]\cap N_{\Gamma_q}^2[  (q,q,0)_{1}] =\{(q,q,q)_1\}\cup \{(0,0,c)_{1}: c\in \mathbb{F}_q\}$ and $N_{\Gamma_q}^2[(q,q,\xi)_{1}]=\bigcup_{j\in\mathbb{F}_q }
N_{\Gamma_q}[(q,\xi,j)_{0}]  \cup N_{\Gamma_q}[(q,q,q)_{0}]$.
 Let us denote by $F=\bigcup_{c\in \mathbb{F}_q}
N_{\Gamma_q}[(q,0,c)_{1}]\cup
N_{\Gamma_q}[(q,q,0)_{1}]\cup \{(0,0,c)_{1}: c\in \mathbb{F}_q\}$. We can check that
$F\cap N_{\Gamma_q}^2[(q,q,\xi)_{1}]=\{(q,q,q)_{0}, (q,q,0)_{1}\}$ (see
Figure \ref{estruc}).
 Hence  $|B|=|N_{\Gamma_q}^2[(q,q,\xi)_{1}]|+|F|-2=  1+(q+1)+q(q+1)+(q+1)(q+2)+q-2=2q^2+6q+2$. Let us prove that $B$ is a perfect dominating set.

For all  vertices  $(x,y,z)_{0}\in V_{0}\setminus B$ with $x\in \mathbb{F}_q\cup \{q\}$, $y,z \in \mathbb{F}_q$ we have:
$$\begin{array}{lll} N_{\Gamma_q}((x,y,z)_0)\cap B&=&N_{\Gamma_q}((x,y,z)_0)\cap N_{\Gamma_q}^2[(q,q,\xi)_{1}]\\[2ex]
&=&\left\{\begin{array}{ll}
\{(\xi,~y-\xi x, ~\xi^2x-2\xi y+z)_{1}\}\subset N_{\Gamma_q}[(q,\xi,y)_{0}]& \mbox{ if } x\ne q;\\
\{(q,q,y)_1\}\subset N_{\Gamma_q}[(q,q,q)_{0}]   & \mbox{ if } x=q.\\
\end{array}\right.\end{array}$$
Moreover, observe that $N_{\Gamma_q}((q,0,c)_{1})\setminus\{(q,q,c)_{0}\}=\{(c,0,j)_{0}:j\in \mathbb{F}_q\}$; and $N_{\Gamma_q}((0,0,c)_{1})=\{(x,0,c)_{0}:x\in \mathbb{F}_q\cup \{q\}\}$, see Figure \ref{estruc}. Then $$F\cap V_{0}=\{(x,0,c)_{0}:x\in \mathbb{F}_q\cup \{q\}, c\in \mathbb{F}_q\}\cup \{(q,q,x)_{0}:x\in \mathbb{F}_q\cup \{q\}\}.$$
 Also, for all  vertices  $(a,b,c)_{1}\in V_{1}\setminus B$ with $a\in \mathbb{F}_q\cup \{q\}$, $b,c \in \mathbb{F}_q$ we have:
$$N_{\Gamma_q}((a,b,c)_1)\cap B=N_{\Gamma_q}((a,b,c)_1)\cap F=\left\{\begin{array}{ll}
\{(-a^{-1}b,0,ab+c)_{0}\}& \mbox{ if } a\ne 0, q;\\
\{(q,0,c)_0\}& \mbox{ if } a=0;\\
\{(q,q,c)_0\}& \mbox{ if } a=q.\\
\end{array}\right.$$
Therefore $B$ is a perfect dominating set of $\Gamma_q$.

$(iii)$ Let denote $R_{0}= \displaystyle \bigcup_{x\in \mathbb{F}_q\cup \{q\}} N_{\Gamma_q}[(q,x,0)_{0}]\cup\left(\bigcap_{x\in \mathbb{F}_q\cup\{q\}} N^2_{\Gamma_q}[(q,x,0)_0]\right)$.
Theorem \ref{main8}, yields that $N_{\Gamma_q}((q,x,0)_{0})\setminus\{(q,q,x)_{1}\}=\{(x,a,0)_{1}:a\in \mathbb{F}_q\}$; and  $N_{\Gamma_q}((0,y,0)_{0})=\{(a,y,-2ay)_{1}:a\in \mathbb{F}_q\}\cup \{(q,y,0)_{1}\}$. Since $q$ is even, $-2ay=0$ and therefore  $$  \bigcup_{x\in \mathbb{F}_q\cup \{q\}} N_{\Gamma_q}\left((q,x,0)_{0}\right)=\bigcup_{y\in \mathbb{F}_q} N_{\Gamma_q}\left((0,y,0)_{0}\right)\cup N_{\Gamma_q}\left((q,q,q)_{0}\right)  \mbox{ (see Figure \ref{estruc2}).}$$
Hence $\displaystyle \bigcap_{x\in \mathbb{F}_q\cup\{q\}} N^2_{\Gamma_q}[(q,x,0)_0]= \{(0,y,0)_{0}: y\in \mathbb{F}_q\}\cup\{(q,q,q)_{0}\}$, implying
 that $|R_{0}|=(q+1)^2+2(q+1)$.

Let  $R_1=\displaystyle \bigcup_{x\in \mathbb{F}_q}N_{\Gamma_q}[(x,x,p(x))_{1}]
     \cup N_{\Gamma_q}[  (q,1,1)_{1}]\cup \left( \bigcap_{x\in \mathbb{F}_q}N_{\Gamma_q}^2[(x,x,p(x))_{1}]\cap N_{\Gamma_q}^2[  (q,1,1)_{1}]\right) $.
     By Theorem \ref{main8},  it is not difficult to check that $\displaystyle  \{ (x,x,p(x) )_{1}  :x\in \mathbb{F}_q \}\cup \{   (q,1,1)_{1}\} $ is a set of $q+1$ vertices mutually at distance four in $\Gamma_q$. Also $\displaystyle  \{ (x,1+x, p(x))_{1}  :x\in \mathbb{F}_q \}\cup \{    (q,0,1)_{1}\} $ is a set of $q+1$ vertices mutually at distance four in $\Gamma_q$.
 Let us show that
\begin{equation}\label{eq1}
     \displaystyle \bigcup_{x\in \mathbb{F}_q}N_{\Gamma_q}\left( (x,x,p(x))_{1} \right) \cup N_{\Gamma_q}\left(    (q,1,1)_{1} \right)  \\
   =\displaystyle \bigcup_{x\in \mathbb{F}_q}N_{\Gamma_q}\left( (x,1+x,p(x))_{1} \right) \cup N_{\Gamma_q}\left(   (q,0,1)_{1} \right) .\end{equation}
 Note that the  sets on  
 both sides have the same cardinality, then to prove the equality it is enough to show one inclusion. We have
 $$\begin{array}{lll} N_{\Gamma_q}((x,x,p(x))_{1})&=&\{(j,xj+x,x^2j+p(x))_{0}:j\in \mathbb{F}_q\}\cup \{(q,x,p(x))_{0}\},  \mbox{for all } x\in \mathbb{F}_q;\\[1ex]
N_{\Gamma_q}((q,1,1)_{1})&=&\{(1, 1, j)_{0}:j\in \mathbb{F}_q\}\cup \{(q,q,1)_{0}\}.
\end{array}$$
Furthermore, since $q$ is even,
$$\begin{array}{lll} N_{\Gamma_q}((x,1+x,p(x))_{1}) &=&\{(j,jx+1+x,x^2j+p(x))_{0}:j\in \mathbb{F}_q\}\cup \{(q,x,p(x))_{0}\},  \mbox{for all } x\in \mathbb{F}_q;\\[1ex]
N_{\Gamma_q}((q,0,1)_{1})&=&\{(1, 0, j)_{0}:j\in \mathbb{F}_q\}\cup \{(q,q,1)_{0}\}.
\end{array}$$
We can check that
$$\begin{array}{lll} N_{\Gamma_q}((x,x,p(x))_{1})\cap N_{\Gamma_q}((x,1+x,p(x))_{1})&=&  \{(q,x,p(x))_{0}\},  \mbox{for all } x\in \mathbb{F}_q;\\[1ex]
N_{\Gamma_q}((q,1,1)_{1})\cap N_{\Gamma_q}((q,0,1)_{1})&=&  \{(q,q,1)_{0}\}.
\end{array}$$
    For all $j\in \mathbb{F}_q$, $j\ne 1$, $(j,xj+x,x^2j+p(x))_{0}\in N_{\Gamma_q}\left( (x,x,p(x))_{1} \right) \cap N_{\Gamma_q}\left( (v,1+v,p(v)  )_{1} \right) $  where $v=(1+j)^{-1}+x$ because $q$ is even. And  $(1,0, x+1)_{0}\in N_{\Gamma_q}\left( (x,x,p(x))_{1} \right) \cap N_{\Gamma_q}\left( (q,0,1)_{1} \right) $  (see Figure \ref{estruc2}) because $p(x)=1+x+x^2$ and $p(x)+x^2=1+x$.
Furthermore,   for all $j\in \mathbb{F}_q$,   $(1,1,j)_{0}\in N_{\Gamma_q}\left( (q,1,1)_{1} \right) \cap N_{\Gamma_q}\left( (a,1+a , 1+a+a^2)_{1} \right) $  where $a= 1+j $.  Hence 
equality (\ref{eq1}) holds. This implies
that $\bigcap_{x\in \mathbb{F}_q}N_{\Gamma_q}^2[(x,x,p(x))_{1}]\cap N_{\Gamma_q}^2[  (q,1,1)_{1}]=
\{(x,1+x,p(x))_{1}: x\in \mathbb{F}_q\} \cup \{(q,0,1)_{1}\}. $ Thus $|R_1|=(q+1)^2+2(q+1)$.

 To finish the proof note that every vertex   $f\in R_r$, $r=0,1$,  with $|N_{\Gamma_q}(f)\cap R_r|=2$ has exactly one neighbor in $R_{r+1}$ and $q-2$ more neighbors in $V(\Gamma_q)\setminus C$. Moreover, every vertex $v\in V(\Gamma_q)\setminus C$ has $|N_{\Gamma_q}(v)\cap C|\le 1$ because the diameter of the subgraph induced by $C=R_0\cup R_1$ is 5 and the girth of $\Gamma_q$ is 8. This implies that $|N_{G_q}(R_0\cup R_1)\cap V(\Gamma_q)\setminus C |=2(q-2)(q+1)^2=2(q^3-3q-2)=|V(\Gamma_q)\setminus C |$
yielding that $|N_{\Gamma_q}(v)\cap C|= 1$ for all $v\in V(\Gamma_q)\setminus C$. Therefore $C$ is a perfect dominating set.
\end{proof}


\begin{lemma}\label{cll1}Let $G_q$ be the graph from Definition \ref{GrGq}. Define $P=\{(q,y,z)_0: y,z\in \mathbb{F}_q, y\not=0,1,\xi\}$ and $R=N_{G_q}(P)\cap N^5_{G_q}((q,1,0)_0)$. Then  $|R|=|P|=q(q-3)$,
 $|N_{G_q}(R) |= 2q( q-2)$ and every $v\in N_{G_q}(R)\setminus P$   has exactly $1$ neighbor   in $N_{G_q}^5((q,1,0)_0)\setminus R$.

\end{lemma}
\begin{proof}
First, note that for all $y\in  \mathbb{F}_q\setminus \{0,\xi\}$ the set of $q$ vertices $\{(q,y,z)_0: z\in \mathbb{F}_q\}$  are mutually at distance 6 in $G_q$ because they were $q$ neighbors in $\Gamma_q$ of the removed vertex $(q,q,y)_1$. Moreover, the vertices $(x,0,z)_0$  with second coordinate zero have been removed from $\Gamma_q$ to obtain $G_q$.   Therefore according to Definition \ref{gammaq}, the paths of length four in $G_q$ joining $(q,1,0)_0$ and a vertex from $P$ are as follows (see Figure \ref{arbol}):

$(q,1,0)_0~(1,b,0)_1(x,x+b,x+2b )_0~ (y, t,z)_1~ (q,y,z)_0$, for all $b,x,t\in \mathbb{F}_q$ such that $b+x\ne 0$.

\noindent Hence $x+b=xy+t$ and $x+2b=y^2x+2yt+z$ for all $b,x,t\in \mathbb{F}_q$ such that $b+x\ne 0$. The claim follows because if $x+b=xy+t=0$, then  $x+2b=y^2x+2yt+z$ gives that  $t=(1-y^2)^{-1}yz$, that is, $(y,(1-y^2)^{-1}yz,z)_1\in R$ is the unique neighbor in $R$ of $(q,y,z)_0\in P$. Therefore every $(q,y,z)_0\in P$ has a unique neighbor $(y,b,z)_1\in R$ yielding that
 $|R|=|P|=q(q-3)$.

It follows that every $v\in N_{G_q}(R)\setminus P$ has at most
 $|R|/q=q-3$ neighbors in $R$ because for each $y$ the vertices from the set $\{(q,y,z)_0:  z \in \mathbb{F}_q \}\subset P$  are mutually at distance 6. Furthermore, every $v\in N_{G_q}(R)\setminus P$ has at most one neighbor in $N_{G_q}^5((q,1,0)_0)\setminus R$ because the vertices $\{(q,1,z)_0: z \in \mathbb{F}_q, z\ne 0\}$ are mutually at distance 6. Therefore every $v\in N_{G_q}(R)\setminus P$   has at least two neighbors  in $ N_{G_q}^3((q,1,0)_0)$. Thus denoting $K=N_{G_q}(N_{G_q}(R)\setminus P)\cap N_{G_q}^3((q,1,0)_0)$ we have
 \begin{equation}\label{auxx1}  |K|\ge  2 |N_{G_q}(R)\setminus P|  .
 \end{equation}

\noindent Moreover, observe that $ (N_{G_q}(P)\setminus R)\cap K =\emptyset$ and since the elements of $P$ are mutually at distance at least 4 we obtain that $|N_{G_q}(P)\setminus R|=q|P|-|R|=(q-1)|P|$. Hence
$$ |N_{G_q}^3((q,1,0)_0)| \ge  | N_{G_q}(P)\setminus R |+|K|
=  (q-1)|P|+|K|.
  $$
Since  $|N_{G_q}^3((q,1,0)_0)|=q(q-1)^2$ and $|P|=q(q-3)$ we obtain that $|K|\le 2q(q-1)$ yielding by (\ref{auxx1}) that  $|N_{G_q}(R)\setminus P|\le q(q-1)$. As $P$ contains at least $q$ elements mutually at distance 6, so $R$ contains at least $q$ elements mutually at distance 4,  thus we have
  $|N_{G_q}(R)\setminus P|\ge q^2-q$. Therefore $|N_{G_q}(R)\setminus P|= q^2-q$ and all the above inequalities are actually equalities. Thus $|N_{G_q}(R) |=  q^2-q+|P|=2q( q-2)$ and every $v\in N_{G_q}(R)\setminus P$   has exactly $1$ neighbor in  $N_{G_q}^5((q,1,0)_0)\setminus R$.
\end{proof}

\begin{proof}{\bf of Theorem \ref{q-1}:} Let $G_q $ be the $q$-regular graph from Definition \ref{GrGq} and consider the sets $P=\{(q,y,z)_0: y,z\in \mathbb{F}_q, y\not=0,1,\xi\}$ and $R=N_{G_q}(P)\cap N^5_{G_q}((q,1,0)_0)$. In particular the vertices of the set $\{(q,1,z)_0:  z\in \mathbb{F}_q\}$ are mutually at distance 6, then by Lemma \ref{cll1}, we have
$$\begin{array}{lll}  | N^4_{G_q}((q,1,0)_0)\setminus N_{G_q}(R)|&=&\displaystyle |\bigcup_{z\in \mathbb{F}_q\setminus \{0\}} (N^2_{G_q}((q,1,z)_0)\cup P)\setminus N_{G_q}(R)|\\[1ex]
&=& q(q-1)^2+q(q-3)-2q(q-2)\\ &=& q(q-1)   (q-2).\end{array}$$
 Let us denote by $E[A,B]$ the set of edges between two set of vertices $A$ and $B$. Then
$|E[N^3_{G_q}((q,1,0)_0),N^4_{G_q}((q,1,0)_0)]|=q(q-1)^3$ and   $|E[N^3_{G_q}((q,1,0)_0),N^4_{G_q}((q,1,0)_0)\setminus N_{G_q}(R)]|=q(q-1)^2(q-2)$. Therefore, $|E[N^3_{G_q}((q,1,0)_0),  N_{G_q}(R)]|=q(q-1)^3-q(q-1)^2(q-2)=q(q-1)^2=|N^3_{G_q}((q,1,0)_0)|$, which  implies that every $v\in N_{G_q}^3((q,1,0)_0)$   has exactly one neighbor   in $N_{G_q}(R) $. It follows   that  $S:=\bigcup_{z\in \mathbb{F}_q} N_{G_q}[(q,1,z)_0]\cup   N_{G_q}[R] $ is a perfect dominating set of $G_q$.
Furthermore, by Lemma  \ref{cll1}, $|S|=q^2+q+ q(3q-7)=4q^2-6q$.
 Therefore a $(q-1)$-regular graph of girth 8  can be obtained by deleting from $G_q$  the indicated perfect dominating set $S$, see Figure \ref{arbol}. This graph  has order $2q(q^2-2)- 2q(2q -3)=2q(q -1)^2 $.
\end{proof}

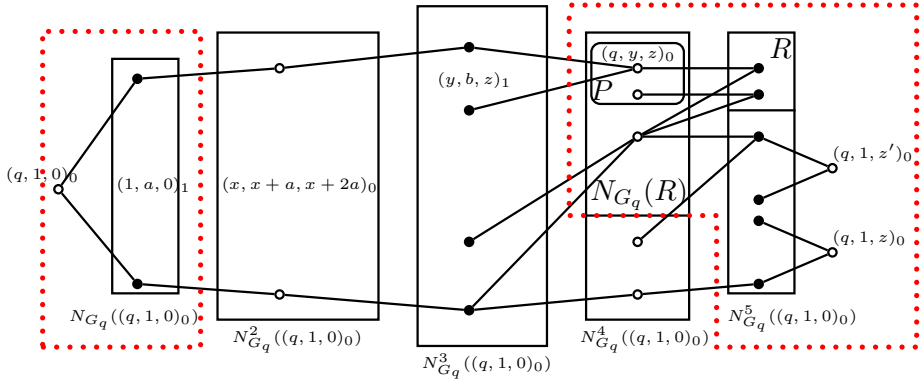
\begin{figure}
\begin{center}
\psset{unit=0.7mm, linewidth=0.03cm}
   \begin{pspicture}(50,-2)(100,77)

\cnode(-5,30){1}{q10}\put(-15,32){\tiny$(q,1,0)_0$}
\cnode*(10,12){1}{a1}   \cnode*(10,51){1}{a2}

\psframe[](5,10)(18,55) \put(6,30){\tiny$(1,a,0)_1$}
\put(-3,5){\tiny$N_{G_q}((q,1,0)_0)$}

\psframe[](25,5)(56,60) \put(26,30){\tiny$(x,x+a,x+2a)_0$}
\cnode(37,10){1}{x1}   \cnode(37,53){1}{x2}
\put(28,1){\tiny$N_{G_q}^2((q,1,0)_0)$}

\psframe[](63,0)(88,65) \put(67,50){\tiny$(y,b,z)_1$}
\cnode*(73,7){1}{y1}  \cnode*(73,45){1}{y0}
\cnode*(73,57){1}{y2}
\cnode*(73,20){1}{y3}
\put(64,-4){\tiny$N_{G_q}^3((q,1,0)_0)$}

 \psframe[](95,5)(115,60) \put(98,55){\tiny$(q,y,z)_0$}
 \cnode(105,53){1}{q1}
 \cnode(105,48){1}{q0}
 \psframe[framearc=0.3](96,46)(114,58)
   \put(96,47){$P$}
 \cnode(105,40){1}{q2}
 \put(96,28){$N_{G_q}(R)$}
 \cnode(105,10){1}{q3}
  \cnode(105,20){1}{q4}
\psline(95,25)(115,25)
 \put(94,1){\tiny$N_{G_q}^4((q,1,0)_0)$}

 \psframe[](122,10)(135,60)
 \cnode*(128,53){1}{r1}  \put(130,55){$R$}
 \psline(122,45)(135,45)
 \put(122,5){\tiny$N_{G_q}^5((q,1,0)_0)$}
 \cnode*(128,48){1}{r0}

 \cnode*(128,12){1}{n1}
 \cnode*(128,24){1}{n2}
 \cnode*(128,28){1}{n3}
 \cnode*(128,40){1}{n4}

 \cnode(142,18){1}{v1}\put(142,20){\tiny$ (q,1,z)_0$}
 \cnode(142,34){1}{v2}  \put(142,36){\tiny$ (q,1,z')_0$}

 \ncline[]{-}{q10}{a1}
 \ncline[]{-}{q10}{a2}
 \ncline[]{-}{a1}{x1}
 \ncline[]{-}{a2}{x2}
 \ncline[]{-}{x1}{y1}
 \ncline[]{-}{x2}{y2}
 \ncline[]{-}{y2}{q1}
 \ncline[]{-}{y0}{q1}
 \ncline[]{-}{y1}{q3}
 \ncline[]{-}{y1}{q2}
 \ncline[]{-}{r1}{q1}
 \ncline[]{-}{r1}{q2}
  \ncline[]{-}{v1}{n1}
 \ncline[]{-}{v1}{n2}
  \ncline[]{-}{v2}{n3}
 \ncline[]{-}{v2}{n4}

  \ncline[]{-}{n4}{q2}
   \ncline[]{-}{n4}{q4}
    \ncline[]{-}{y3}{q2}
 \ncline[]{-}{q3}{n1}
  \ncline[]{-}{q0}{r0}
   \ncline[]{-}{q2}{r0}

 \psset{dotsize=4pt 0}
\psline[linestyle=dotted,linecolor=red,linewidth=2pt](92,65)(158,65)(158,0)(120,0)(120,25)(92,25)(92,65)
\psline[linestyle=dotted,linecolor=red,linewidth=2pt](-8,0)(-8,60)(22,60)(22,0)(-8,0)
%



%
%
%


      \end{pspicture}
\caption{Structure of the graph $G_q$. The eliminated vertices are inside the dotted box. \label{arbol}}
\end{center}
\end{figure}


\bigskip







\end{document}